\documentclass[reqno, 12pt]{amsart} 
\usepackage{graphicx}
\usepackage{amsmath,amssymb,amsthm, mathrsfs, mathtools, booktabs,bm}
\usepackage[foot]{amsaddr}
\usepackage{xcolor}
\usepackage[backref]{hyperref}
\hypersetup{
    colorlinks,
    linkcolor={red!60!black},
    citecolor={green!60!black},
    urlcolor={blue!60!black}
}
\usepackage[T1]{fontenc}
\usepackage{lmodern}
\usepackage[babel]{microtype}
\usepackage[english]{babel}
\usepackage{comment}

\usepackage{geometry}
\numberwithin{equation}{section}
\numberwithin{figure}{section}

\usepackage{enumerate}
\usepackage{enumitem}

\theoremstyle{plain}
\newtheorem{thm}{Theorem}[section]

\newtheorem{remark}[thm]{Remark}

\newtheorem{lemma}[thm]{Lemma}

\DeclarePairedDelimiter{\parens}{(}{)}
\DeclarePairedDelimiter{\set}{\{}{\}}

\DeclarePairedDelimiter\size{\lvert}{\rvert}   

\newcommand{\Han}{H{\`a}n}

\newcommand{\Rodl}{R\"{o}dl}

\newcommand{\Szabo}{Szab\'o}

\newcommand{\Dudek}{Dudek}
\newcommand{\FGLPS}{Fox, Grinshpun, Liebenau, Person, and \Szabo{}}

\renewcommand{\leq}{\leqslant}
\renewcommand{\geq}{\geqslant}
\renewcommand\le{\leqslant}
\renewcommand\ge{\geqslant}

\newcommand{\calP}{\ensuremath{\mathcal{P}}}
\newcommand{\calL}{\ensuremath{\mathcal{L}}}
\newcommand{\calI}{\ensuremath{\mathcal{I}}}

\title[New upper bound on the minimum degree of minimal Ramsey graphs]{A new upper bound on the minimum degree of minimal Ramsey graphs}

\author{Anurag Bishnoi$^1$}
\address[Bishnoi]{$^1$ Delft Institute of Applied Mathematics, TU Delft, Netherlands.}
\email[Bishnoi]{A.Bishnoi@tudelft.nl}
\thanks{Anurag Bishnoi's research supported by a Discovery Early Career Award of the Australian Research Council (No.~DE190100666)}

\author{Thomas Lesgourgues$^2$}
\address[Lesgourgues]{$^2$ University of Waterloo, Canada.}
\email[Lesgourgues]{tlesgourgues@uwaterloo.ca}
\thanks{Thomas Lesgourgues' research supported by an Australian Government Research Training Program Scholarship and the School of Mathematics and Statistics, UNSW}
\thanks{}

\begin{document}

\begin{abstract}
    We prove that $s_r(K_{k+1}) = O(k^3 r^3 \ln^3 k)$, where $s_r(K_k)$ is the Ramsey parameter introduced by Burr, Erd\H{o}s and Lov\'{a}sz 
    in 1976, which is defined as the smallest minimum degree of a graph $G$ such that any $r$-colouring of the edges of $G$ contains a monochromatic $K_k$, whereas no proper subgraph of $G$ has this property. 
\end{abstract}

\maketitle 
\section{Introduction}
A graph $G$ is called $r$-Ramsey for another graph $H$, denoted by $G \rightarrow (H)_r$, if every $r$-colouring of the edges of $G$ contains a monochromatic copy of $H$. 
Observe that if $G \rightarrow (H)_r$, then every graph containing $G$ as a subgraph is also $r$-Ramsey for $H$. 
Some very interesting questions arise when we study graphs $G$ which are minimal with respect to $G \rightarrow (H)_r$, that is, $G \rightarrow (H)_r$ but there is no proper subgraph $G'$ of $G$ such that $G' \rightarrow (H)_r$. 
We call such graphs \textit{$r$-Ramsey minimal for $H$} and we denote the set of all $r$-Ramsey minimal graphs for $H$ by $\mathcal{M}_r(H)$. It follows from the classical result of Ramsey \cite{Ramsey:1929aa} that $\mathcal{M}_r(H)$ is non-empty for any choices of graph $H$ and positive integer $r$. 

Many questions on $\mathcal{M}_r(H)$ have been explored; for example, the Ramsey number $R_r(H)$ denotes the smallest number of vertices of any graph in $\mathcal{M}_r(H)$ and the size Ramsey number $\hat{R}_r(H)$ denotes the smallest number of edges. We refer the reader to \cite{Burr81, Burr85, Luczak94, Rodl08} for various results on Ramsey minimal problems. In this paper, we will be interested in the \textit{smallest minimum degree of an $r$-Ramsey minimal graph}, defined by 
\[s_r(H) \coloneqq \min_{G \in \mathcal{M}_r(H)} \delta(G),\]
for a finite graph $H$ and positive integer $r$, where $\delta(G)$ denotes the minimum degree of $G$. Trivially, we have $s_r(H) \leq R_r(H) - 1$, since the complete graph on $R_r(H)$ vertices is $r$-Ramsey for $H$ and is $(R_r(H) - 1)$-regular (taking minimal Ramsey subgraphs of this graph cannot increase the minimum degree). This parameter was introduced by Burr, Erd\H{o}s and Lov\'{a}sz \cite{Burr:1976aa} in 1976. They were able to show the rather surprising exact result, $s_2(K_{k+1}) = k^2$, where $K_{k+1}$ is the complete graph on $k+1$ vertices, which is far away from the trivial exponential bound of $s_2(K_{k+1}) \leq R_2(k+1) - 1$.

While no precise values are known for $s_r(K_{k+1})$ for $r>2$,~\FGLPS{}~\cite{Fox:2016aa} showed that $s_r(K_{k+1})$ is quadratic in $r$, up to a polylogarithmic factor, when the size of the clique is fixed. Formally, they showed that  for all $k \geq 2$ there exist constants $c_k, C_k > 0$ such that for all $r \geq 3$, we have
\begin{equation}
    c_k r^2 \frac{\ln r}{ \ln{\ln r}} \le s_r(K_{k+1}) \le C_k r^2 (\ln r)^{8k^2}.\label{eq:FGLPS}
\end{equation}

When $k=2$, Guo and Warnke~\cite{Guo-Warnke20} settled the exact polylogarithmic factor, following earlier work in~\cite{Fox:2016aa}. In the other regime, when the number of colours is fixed, \Han{}, \Rodl{}, and  \Szabo{}~\cite{Han:2018aa} showed that $s_r(K_{k+1})$ is quadratic in the clique size $k$, up to a polylogarithmic factor. They proved that, for every integer $r \geq 2$ there exists a constant $C'_r$ such that for every integer $k \geq 3$
\begin{equation}
     s_r(K_{k+1})\leq C'_r(k\ln k)^2.\label{eq:HRS_asymp}
\end{equation}

The constant in the upper bound of \eqref{eq:FGLPS} is rather large ($C_k\sim k^2 2^{8k^2}$), and in particular not polynomial in $k$. To remedy this, \FGLPS{}~\cite{Fox:2016aa} also proved an upper bound which is polynomial in both $k$ and $r$ and is applicable for small values of $r$ and $k$.

\begin{thm}[Fox, Grinshpun, Liebenau, Person, Szab\'{o}]
\label{thm:FGLPS}
For all $k\ge 2$, $r \ge 3$, $s_r(K_{k+1}) \le 8k^6 r^3$. 
\end{thm}

\Han{}, \Rodl{}, and  \Szabo{}~\cite{Han:2018aa} further proved that the constant $C'_r$ of~\eqref{eq:HRS_asymp} is polynomial in $r$ when $r<k^2$ and $k$ is large enough. They showed that there exists a constant $k_0$ such that for every $k > k_0$ and $r < k^2$, we have $s_r(K_{k+1}) \leq 80^3 (r \ln r)^3 (k \ln k)^2$. Combining with~\eqref{eq:FGLPS}, this result implies the existence of a large absolute constant $C$ and a polynomial upper bound for $s_r(K_{k+1})$.

\begin{thm}[H\`{a}n, R\"{o}dl, Szab\'{o}]
\label{thm:HRS}
There exists an absolute constant $C$ such that for every $k\ge 2$ and $r < k^2$,
\[s_r(K_{k+1}) \leq C (r \ln r)^3 (k \ln k)^2.\]
\end{thm}

Finally, using a group theoretic model of generalised quadrangles introduced by Kantor in 1980~\cite{Kantor:1986aa}, Bamberg and the authors~\cite{bambergMinimumDegreeMinimal} proved another polynomial bound, reducing the dependency in $r$, and improving on Theorem~\ref{thm:FGLPS} for any $k,r$ and on Theorem~\ref{thm:HRS} when $r>k^6$.

\begin{thm}[Bamberg, Bishnoi, Lesgourgues]\label{thm:BBL}
There exists an absolute constant $C$ such that for all $k\ge 2$, $r \ge 3$,  $s_r(K_{k+1}) \leq C k^5 r^{5/2}$. 
\end{thm}

These theorems all use the equivalence between $s_r(K_k)$ and another extremal function, called the \textit{$r$-colour $k$-clique packing number} \cite{Fox:2016aa}. Theorems~\ref{thm:FGLPS} and~\ref{thm:BBL} further use some `triangle-free' point-line geometries, for which, under certain conditions on their parameters, any packing of these geometries implies an upper bound on the $r$-colour $k$-clique packing number. This argumentation, initially developed by \Dudek{} and \Rodl{} \cite{Dudek:2011aa} and then by Fox et al. in~\cite{Fox:2016aa}, has been further optimized by Bamberg et al. in~\cite[Lemma 3.1]{bambergMinimumDegreeMinimal}. Using this optimized argumentation from~\cite{bambergMinimumDegreeMinimal} and the finite geometric construction of Fox et al. from~\cite{Fox:2016aa}, we show the following general upper bound that improves on the best known bounds for $k\geq 8$ and $r$ in the range $k^2\leq r\leq O(k^4/\ln^6k)$.

\begin{thm}\label{thm:main}
For all $k\ge 2$, $r \ge 3$,  $s_r(K_{k+1}) \leq (8kr\ln k)^{3}$. 
\end{thm}

In Section \ref{sec:best_possible}, we then proceed to show that this upper bound for $s_r(K_{k+1})$ is in some sense the `best possible' bound one can obtain using triangle-free partial linear spaces and Lemma \ref{lemma:Bound_s}.

\section{Packing partial linear spaces}\label{sec:background}

A partial linear space is an incidence structure of points $\mathcal{P}$ and lines $\mathcal{L}$, with an incidence relation such that there is at most one line through every pair of distinct points. If every line is incident with exactly $s+1$ points and every point is incident with exactly $t+1$ lines, then the partial linear space has order $(s,t)$. If there are no three distinct lines pairwise meeting each other in three distinct points, then the partial linear space is \textit{triangle-free}. \emph{Generalised quadrangles} are standard examples of triangle-free partial linear spaces, with the additional property that for every non-incident point-line pair $x,\ell$ there exists a unique point $x'$ incident to $\ell$ such that $x$ and $x'$ are collinear (see the book by Payne and Thas \cite{Payne:2009aa}  for a standard reference on finite generalised quadrangles).

The next lemma can be found in \cite[Lemma 3.1]{bambergMinimumDegreeMinimal}. Its proof follows a methodology initially developed by \Dudek{} and \Rodl{} \cite{Dudek:2011aa}, using the $r$-colour $k$-clique packing number developed in \cite{Fox:2016aa}.

\begin{lemma}[Bamberg, Bishnoi, Lesgourgues]\label{lemma:Bound_s}
Let $r, k, s, t$ be positive integers. 
Say there exists a family $(\calI_i)_{i=1}^r$ of triangle-free partial linear spaces of order $(s, t)$, on the same point set $\calP$ and with pairwise disjoint line sets $\calL_1, \dots, \calL_r$, such that the point-line geometry $\parens*{\calP, \bigcup_{i=1}^r \calL_i}$ is also a partial linear space. If $s \geq 3rk\ln k$ and $t  \geq 3k(1+\ln r)$, then $s_r(K_{k+1}) \leq \size{\calP}$. 
\end{lemma}

The following lemma, that will imply Theorem \ref{thm:main},
is a reformulation in the language of (triangle-free) partial linear space of a construction that can be found in \cite[Proof of Lemma 4.4]{Fox:2016aa}. We include the proof for completeness.

\begin{lemma}\label{lemma:exist_PLS}
Let $q$ be any prime power. There exists a family $(\calI_i)_{i=1}^{q-1}$ of triangle-free partial linear spaces of order $(q-1, q-2)$, on the same point set $\calP$ and with pairwise disjoint line-sets $\calL_1, \dots, \calL_{q-1}$, such that the point-line geometry $\parens*{\calP, \bigcup_{i=1}^{q-1} \calL_i}$ is also a partial linear space.
\end{lemma}

\begin{proof}
Let $\mathbb{F}_q$ be the finite field of order $q$ and for $\lambda \in\mathbb{F}_q\setminus\set{0}$ let $M_\lambda$ be the \textit{$\lambda$-moment curve},

\[ M_\lambda = \set*{(1,\lambda\alpha,\lambda\alpha^2): \alpha\in\mathbb{F}_q\setminus\set{0}}.\]

Note that for non-zero $\lambda_1 \neq \lambda_2$ the two curves $M_{\lambda_1}$ and $M_{\lambda_2}$ do not intersect. A line in $\mathbb{F}^3_q$ is a set of the form $\ell_{\bm{s,v}}=\set*{\beta\bm{s}+\bm{v}:\ \beta\in\mathbb{F}_q}$, where $\bm{s}\in\mathbb{F}^3_q\setminus\set{0}$ is the slope of the curve. For $\lambda \in\mathbb{F}_q\setminus\set{0}$, we define the incidence structure $\calI_\lambda = (\mathbb{F}^3_q,\calL_\lambda)$ where $\calL_\lambda$ is the set of lines with slope from the $\lambda$-moment curves, i.e.
\[ \calL_\lambda :=\set{\ell_{\bm{s,v}}: \bm{s}\in M_\lambda,\bm{v}\in\mathbb{F}^3_q}.\]

Fox et al.~\cite{Fox:2016aa} established the following properties about each structure $\calI_\lambda$, $\lambda\in\mathbb{F}_q\setminus\set{0}$.

\begin{enumerate}
    \item $\calI_\lambda$ is a partial linear space (any two lines meet in at most one point).
    \item Every line $\ell\in\calL_\lambda$ contains $q$ points and every point $\bm{v} \in \mathbb{F}^3_q$ is contained in $q - 1$ lines.
    \item $\calI_\lambda$ is triangle-free. No three lines in $\calL_\lambda$ intersect pairwise in three distinct points.
\end{enumerate}

Further, they proved that for $\lambda_1\neq \lambda_2$,
\begin{enumerate}[resume]
    \item $\calL_{\lambda_1}\cap\calL_{\lambda_2}=\emptyset$.
\end{enumerate}

Given that any line in $\bigcup_{i=1}^{q-1} \calL_i$ is a line of the affine space $\mathbb{F}^3_q$, Property (4) is sufficient to deduce that any two lines in $\bigcup_{i=1}^{q-1} \calL_i$ meet in at most one point. Therefore the point-line geometries $\calI_\lambda$ are such that $\parens*{\calP, \bigcup_{i=1}^{q-1} \calL_i}$ is also a partial linear space.
\end{proof}

Theorem~\ref{thm:main} is a direct consequence of Lemma~\ref{lemma:exist_PLS}.

\begin{proof}[Proof of Theorem~\ref{thm:main}]
Let $k\ge 2$, $r \ge 3$, and let $q$ be the smallest prime such that $q\geq 4kr\ln k$. By Bertrand's postulate, $q\leq 8kr\ln k$. By Lemma \ref{lemma:exist_PLS}, there exists a family of $r<q$ triangle-free partial linear spaces of order $(q-1, q-2)$, on the same point set $\calP$ and pairwise disjoint line-sets $\calL_1, \dots, \calL_r$, such that the point-line geometry $\parens*{\calP, \bigcup_{i=1}^r \calL_i}$ is also a partial linear space. Note that with $k\ge 2$ and $r \ge 3$, we have $q-1 \geq 3rk\ln k$ and $q-2 \geq 3k(1+\ln r)$. By Lemma \ref{lemma:Bound_s}, $s_r(K_{k+1})\leq \size{\calP}$, and then $\size{\calP}= q^3$ yields the desired bound.
\end{proof}

\begin{remark}
Each point-line geometry $(\mathcal{P}, \mathcal{L}_\lambda)$ in the construction above is a subgeometry of a $T_2(O)$ generalized quadrangle (see  Section 3.1.2 in \cite{Payne:2009aa}). 
\end{remark}

A careful review of the arguments in \cite[Lemma 3.1 and 5.2]{bambergMinimumDegreeMinimal} would allow a small optimisation on the multiplicative constant of this corollary. However in light of the conjectured quadratic upper bound \cite[Conjecture 5.2]{bambergMinimumDegreeMinimal}, we did not push this further.

\section{Best possible total degree}\label{sec:best_possible}

We now prove that the upper bound from Theorem \ref{thm:main} is the best possible general polynomial bounds using triangle-free partial linear spaces and Lemma \ref{lemma:Bound_s}. 

\begin{thm}\label{thm:best_possible}
For some $\alpha\geq 1$, assume that for any positive integers $k,r$ there exist a family $(\calP,\calL_i)$ of triangle-free partial linear spaces of order $(q,q^\alpha)$, satisfying the assumptions of Lemma \ref{lemma:Bound_s}. Then 
$\size{\calP}=\Omega(k^{2+\alpha}r^{2+\alpha})$. Similarly if the partial linear spaces are of order $(q^\alpha,q)$ then $\size{\calP}=\Omega(k^{2\alpha+1}r^{2+1/\alpha})$.
\end{thm}

Note that Theorem \ref{thm:main} corresponds to the case $\alpha = 1$. It follows from this proposition that, ignoring polylogarithmic factors, any polynomial upper bound for $s_r(K_{k})$ achieved through Lemma \ref{lemma:Bound_s} must have total degree at least $6$, with $k^3r^3$ being the unique polynomial of degree $6$ possible (up to polylogarithmic factors). 

The proposition is a direct consequence of the following lemma, that yields a lower bound on the number of points in any triangle-free partial linear space. 

\begin{lemma}\label{lemma:counting_points}
Let $(\calP,\calL)$ be a triangle-free partial linear space of order $(s,t)$. Then $\size{\calP}\geq(st+1)(s+1)$.
\end{lemma}

\begin{proof}

For any point $x\in\calP$, let $N(x)$ be the \textit{neighbourhood} of $x$, the set of points $v$ such that $x$ and $v$ are collinear. Any point $x\in\calP$ is incident with $t+1$ lines, each of which contains $s+1$ points including $x$, therefore $|N(x)|=s(t+1)$.

Let $\ell\in\calL$.  Given that the partial linear space is triangle-free, the set of points not incident to $\ell$ contains the disjoint union of $N(x)\setminus\ell$, for $x\in\ell$,  so, 
\[\size{\calP}\geq (s+1) + \sum_{x\in\ell}\size{N(x)\setminus\ell}=(s+1)+(s+1)(s(t+1)-s)=(s+1)(st+1).  \]

Equality is attained if and only if any point not incident to $\ell$ is in $N(x)$ for some $x\in\ell$, meaning that the partial linear space is a generalized quadrangle.
\end{proof}

\begin{proof}[Proof of \ref{thm:best_possible}]
Assume first that for any positive integers $k,r$, the partial linear spaces have order $(q,q^\alpha)$ for $q=q(k,r)$. Then given the condition of Lemma~\ref{lemma:Bound_s}, $q(k,r) \geq 3rk\ln k$ and $q(k,r)^\alpha\geq 3k(1+\ln r)$. As these inequalities have to be verified for any $k,r$, then $q(k,r)$ has to be at least linear in $r$ and linear in $k$, i.e. $q(k,r)=\Omega(rk)$. Then Lemma \ref{lemma:counting_points} yields $\size{\calP}>s^2t=q^{2+\alpha}=\Omega(k^{2+\alpha}r^{2+\alpha})$. 

Similarly, if the partial linear spaces have order $(q^\alpha,q)$ for some $\alpha\geq1$, the condition of Lemma~\ref{lemma:Bound_s} implies $q^\alpha \geq 3rk\ln k$ and $q\geq 3k(1+\ln r)$ and then $q=\Omega(kr^{1/\alpha})$. Lemma \ref{lemma:counting_points} yields $\size{\calP}>s^2t=q^{2\alpha+1}=\Omega(k^{2\alpha+1}r^{2+1/\alpha})$.

\end{proof}

\bibliographystyle{abbrv}
\bibliography{references}
\end{document}